\documentclass[a4paper]{amsart}

\usepackage{
    amsmath,
    amsthm,
    amssymb,
    mathrsfs,
    tikz-cd,
}

\usepackage[T1]{fontenc}
\usepackage{lmodern}

\usepackage{imakeidx}

\usepackage[colorlinks=true,allcolors=blue]{hyperref}
\newcommand{\msc}[1]{\href{https://zbmath.org/classification/?q=#1}{#1}}
\usepackage{orcidlink}

\usepackage{cleveref}
\crefdefaultlabelformat{#2\textup{#1}#3} 

\usepackage[backend=biber,giveninits=true,uniquename=false,doi=true]{biblatex}
\renewbibmacro{in:}{}
\DeclareFieldFormat{pages}{#1}

\usepackage[pagewise]{lineno}

\theoremstyle{plain}
\newtheorem{main} {Theorem}

\newtheorem{thm} {Theorem} [section]

\newtheorem{cor}    [thm] {Corollary}
\newtheorem{lem}    [thm] {Lemma}
\theoremstyle{definition}

\newtheorem{exm}    [thm] {Example}
\newtheorem{rmk}    [thm] {Remark}

\crefname{main}{Theorem}{Theorems}
\crefname{thm}{Theorem}{Theorems}
\crefname{lem}{Lemma}{Lemmas}
\crefname{prop}{Proposition}{Propositions}
\crefname{cor}{Corollary}{Corollaries}
\crefname{rmk}{Remark}{Remarks}

\DeclareMathOperator{\Hom}{Hom}
\DeclareMathOperator{\End}{End}
\newcommand{\ISHom}{\overline{\Hom}}
\newcommand{\PSHom}{\underline{\Hom}}
\newcommand{\PHom}{\mathscr{P}}
\newcommand{\IHom}{\mathscr{I}}
\newcommand{\PCat}{\mathscr{P}}
\newcommand{\ICat}{\mathscr{I}}
\DeclareMathOperator{\Irr}{Irr}
\newcommand{\FF}{F}
\newcommand{\mult}[2]{\left[\, {#1} \mid {#2} \,\right]} 
\newcommand{\ProjFree}{\Omega^0}
\newcommand{\InjFree}{\Sigma^0}
\renewcommand{\phi}{\varphi}

\begin{document}

\title[Reciprocity for projective and injective direct summands]{Reciprocity for projective and injective\\ direct summands}

\author[S.~Koshitani]{Shigeo Koshitani\,\orcidlink{0000-0002-0703-0363}}
\address[Shigeo Koshitani]
{Department of Mathematics and Informatics, Graduate School of Science, Chiba University, 1-33 Yayoi-cho, Inage-ku, Chiba-shi, Chiba, 263-8522, Japan.}
\email{koshitan@math.s.chiba-u.ac.jp}

\author[T.~Sakurai]{Taro Sakurai\,\orcidlink{0000-0003-0608-1852}}
\address[Taro Sakurai]
{Department of Mathematics and Informatics, Graduate School of Science, Chiba University, 1-33 Yayoi-cho, Inage-ku, Chiba-shi, Chiba, 263-8522, Japan.}
\email{tsakurai@math.s.chiba-u.ac.jp}


\subjclass[2020]{
    Primary \msc{16G10};
    Secondary  \msc{20C20}, \msc{16D70}, \msc{18G65}%
}

\keywords{Robinson reciprocity, projective cover, injective hull, simple module, tensor-hom adjunction, stable category, costable category, skew bimodule}

\date{\today}

\begin{abstract}
    A reciprocity between multiplicities of projective direct summands for modular group algebras has been known for decades.
    Recently, the first author extended it to finite-dimensional symmetric algebras.
    We further generalize the reciprocity theorem to arbitrary finite-dimensional algebras by establishing a tensor-hom adjunction for stable and costable categories.
\end{abstract}

\maketitle

\setcounter{tocdepth}{2}

\section{Introduction}
Robinson reciprocity is an analog of Frobenius reciprocity in the modular representation theory of finite groups.
Let $\FF$ be an algebraically closed field, $G$ a finite group and $H$ a subgroup of $G$.
\index{$F$ @ $\FF$ : Field}
\index{$G$ : Finite Group}
\index{$H$ : Finite Group}
Consider the projective covers $P(S)$ and $P(T)$ of a simple right $\FF H$-module $S$ and a simple right $\FF G$-module $T$, respectively.
\index{$P(S)$ : Projective Module}
\index{$P(T)$ : Projective Module}
\index{$S$ : Simple Module}
\index{$T$ : Simple Module}
Robinson reciprocity \cite[Theorem 3]{Ro89} relates the multiplicity $\mult{P(T)}{S^G}$ of $P(T)$ as a direct summand of the induction $S^G$ to the multiplicity $\mult{P(S)}{T_H}$ of $P(S)$ as a direct summand of the restriction $T_H$ as
\[
    \mult{P(T)}{S^G} = \mult{P(S)}{T_H}.
\]
\index{$S^G$ : Module}
\index{$T_H$ : Module}
\index{$\mult{P(T)}{S^G}$ : Natural Number}
\index{$\mult{P(S)}{H_H}$ : Natural Number}
This is useful for computing induction and restriction of simple modules over modular group algebras.
Recently, the first author extended it to finite-dimensional symmetric algebras \cite{Ko26}.
In this paper, we further generalize the reciprocity theorem to arbitrary finite-dimensional algebras (\cref{main:reciprocity}) by establishing a tensor-hom adjunction for stable and costable categories (\cref{main:tensor-hom}).
We expect that our generalizations will be useful in other contexts as well.

Let $A$ and $B$ be finite-dimensional algebras over a field $\FF$.
\index{$A$ : Algebra}
\index{$B$ : Algebra}
\index{$F$ @ $\FF$ : Field}
All modules are assumed to be unital and finitely generated.
The multiplicity of a right $A$-module $X$ as a direct summand of a right $A$-module $Y$ is denoted by
\[
    \mult{X}{Y},
\]
which is well-defined by the Krull--Schmidt theorem.
\index{$X$ : Module}
\index{$Y$ : Module}
\index{$\mult{X}{Y}$ : Natural Number}
The injective hull and the projective cover of $X$ are denoted by $I(X)$ and $P(X)$, respectively.
\index{$I(X)$ : Injective Module}
\index{$P(X)$ : Projective Module}
The standard duality is denoted by $({-})^* = \Hom_\FF({-}, \FF)$.
\index{$({-})^*$ : Functor}
An $(A, B)$-bimodule $M$ is called perfect if it is a projective left $A$-module and a projective right $B$-module.
We call $M$ \emph{skew} if:
\begin{itemize}
    \item $M \otimes_B B^*$ is a projective left $A$-module; and
    \item $A^* \otimes_A M$ is a projective right $B$-module.
\end{itemize}
\index{$M$ : Bimodule}
(See also \cref{rmk:skew}.)
A trivial example is $M = A \otimes_\FF B$, which is both perfect and skew.
We can now state our main results.

\begin{main}
    \label{main:reciprocity}
    Let $A$ and $B$ be finite-dimensional algebras over a splitting field and let $M$ be a skew $(A, B)$-bimodule.
    If $S$ is a simple right $A$-module and $T$ is a simple right $B$-module, then
    \[
        \mult{P(T)}{S \otimes_A M} = \mult{I(S)}{\Hom_B(M, T)}.
    \]
\end{main}

To prove this reciprocity, we establish a tensor-hom adjunction for stable and costable categories.
As usual, we write the hom spaces of stable and costable categories as
\begin{align*}
    \PSHom_B({-}, {-}) &= \Hom_B({-}, {-})/\PHom_B({-}, {-}), \\
    \ISHom_A({-}, {-}) &= \Hom_A({-}, {-})/\IHom_A({-}, {-}),
\end{align*}
where $\PHom_B({-}, {-})$ and $\IHom_A({-}, {-})$ denote the subspaces of homomorphisms factoring through projective and injective modules, respectively.
\index{$Hom_B({-}, {-})$ @ $\PSHom_B({-}, {-})$ : Vector Space}
\index{$P_B({-}, {-})$ @ $\PHom_B({-}, {-})$ : Vector Space}
\index{$Hom_A({-}, {-})$ @ $\ISHom_A({-}, {-})$ : Vector Space}
\index{$I_A({-}, {-})$ @ $\IHom_A({-}, {-})$ : Vector Space}
To the best of our knowledge, the following adjunction for arbitrary finite-dimensional algebras is novel.
(See also \cref{rmk:self-injective}.)
\begin{main}
    \label{main:tensor-hom}
    Let $A$ and $B$ be finite-dimensional algebras over a field and let $M$ be a skew $(A, B)$-bimodule.
    If $X$ is a right $A$-module and $Y$ is a right $B$-module, then there exists an isomorphism of vector spaces
    \[
        \PSHom_B(X \otimes_A M, Y)
        \cong \ISHom_A(X, \Hom_B(M, Y))
    \]
    that is natural in $X$ and $Y$.
\end{main}

\section{Preliminaries}
Throughout this section, let $A$ and $B$ be finite-dimensional algebras over a field $\FF$.
In particular, we do not assume that $\FF$ is a splitting field for $A$ or $B$.

\subsection*{Projective and injective modules}
We need lemmas on projective and injective modules concerning skew bimodules to prove \cref{main:tensor-hom}.
\begin{lem}
    \label{lem:tensor}
    Let $M$ be a skew $(A, B)$-bimodule.
    If $J$ is an injective right $A$-module, then $J \otimes_A M$ is a projective right $B$-module.
\end{lem}
\index{$J$ : Injective Module}
\begin{proof}
    Since $J$ is injective,
    there exists a right $A$-module $K$ such that $J \oplus K \cong A^* \oplus \dotsb \oplus A^*$ by \cite[Lemma I.8.17]{SY11}.
    \index{$K$ : Injective Module}
    Then the right $B$-module $J \otimes_A M$ is a direct summand of
    \begin{align*}
        (J \otimes_A M) \oplus (K \otimes_A M)
        &\cong (J \oplus K) \otimes_A M \\
        &\cong (A^* \oplus \dotsb \oplus A^*) \otimes_A M \\
        &\cong (A^* \otimes_A M) \oplus \dotsb \oplus (A^* \otimes_A M).
    \end{align*}
    Since $A^* \otimes_A M$ is a projective right $B$-module, it is a direct summand of a free $B$-module and so is $J \otimes_A M$.
    Hence $J \otimes_A M$ is a projective right $B$-module.
\end{proof}

\begin{lem}
    \label{lem:hom}
    Let $M$ be a skew $(A, B)$-bimodule.
    If $Q$ is a projective right $B$-module, then $\Hom_B(M, Q)$ is an injective right $A$-module.
\end{lem}
\index{$Q$ : Projective Module}
\begin{proof}
    Since $Q$ is projective,
    there exists a right $B$-module $R$ such that $Q \oplus R \cong B \oplus \dotsb \oplus B$.
    \index{$R$ : Projective Module}
    Then the right $A$-module $\Hom_B(M, Q)$ is a direct summand of
    \begin{align*}
        \Hom_B(M, Q) \oplus \Hom_B(M, R)
        &\cong \Hom_B(M, Q \oplus R) \\
        &\cong \Hom_B(M, B \oplus \dots \oplus B) \\
        &\cong \Hom_B(M, B) \oplus \dots \oplus \Hom_B(M, B).
    \end{align*}
    Observe that $\Hom_B(M, B)$ is an injective right $A$-module since $M \otimes_B B^*$ is a projective left $A$-module,
    \cite[Proposition I.8.16(ii)]{SY11} and
    \[
        \Hom_\FF(M \otimes_B B^*, \FF)
        \cong \Hom_B(M, \Hom_\FF(B^*,\FF))
        \cong \Hom_B(M, B)
    \]
    as right $A$-modules.
    Since $\Hom_B(M, B)$ is an injective right $A$-module, it is a direct summand of the direct sum of finitely many copies of $A^*$ and so is $\Hom_B(M, Q)$ by \cite[Lemma I.8.17]{SY11}.
    Hence $\Hom_B(M, Q)$ is an injective right $A$-module.
\end{proof}

\begin{rmk}
    \label{rmk:skew}
    Let $\ICat(A)$ and $\PCat(B)$ denote the categories of injective right $A$-modules and projective right $B$-modules, respectively.
    \index{$I(A)$ @ $\ICat(A)$ : Category}
    \index{$P(B)$ @ $\PCat(B)$ : Category}
    If an $(A, B)$-bimodule $M$ induces two functors
    \[
        \begin{tikzcd}[column sep=huge]
            \PCat(B) \arrow[r, shift left, "{\Hom_B(M, {-})}"] & \ICat(A) \arrow[l, shift left, "{{-} \otimes_A M}"]
        \end{tikzcd}
    \]
    as in \cref{lem:tensor,lem:hom}, then taking $J = A^*$ and $Q = B$ shows that $M$ is skew.
\end{rmk}

\subsection*{Simple modules}
We need lemmas on simple modules in stable and costable categories to prove \cref{main:reciprocity}.
For a right $A$-module $X$, let $\InjFree(X)$ denote the injective-free direct summand of $X$ and let $\ProjFree(X)$ denote the projective-free direct summand of $X$.
\index{$\InjFree(X)$ : Module}
\index{$\ProjFree(X)$ : Module}
These are uniquely determined up to isomorphism by the Krull--Schmidt theorem.

\begin{lem}
    \label{lem:stable}
    If $Y$ is a right $B$-module and $T$ is a simple right $B$-module, then
    \[
        \PSHom_B(Y, T) \cong \Hom_B(\ProjFree(Y), T).
    \]
\end{lem}
\begin{proof}
    From the definition of $\PSHom$ and $\ProjFree$,
    \[
        \PSHom_B(Y, T) \cong \PSHom_B(\ProjFree(Y), T).
    \]
    Hence it suffices to prove that the canonical epimorphism
    \[
        \Hom_B(\ProjFree(Y), T) \to \PSHom_B(\ProjFree(Y), T)
    \]
    is a monomorphism.

    Let $\alpha \colon \ProjFree(Y) \to T$ factor through a projective right $B$-module $Q$ as $\alpha = \gamma \circ \beta$.
    \index{$\alpha$ : Homomorphism}
    \index{$\beta$ : Homomorphism}
    \index{$\gamma$ : Homomorphism}
    Assume that $\alpha$ is non-zero to obtain a contradiction.
    Let $\pi \colon P(T) \to T$ be the projective cover of $T$.
    \index{$\pi$ : Homomorphism}
    Since $Q$ is projective, there exists $\phi \colon Q \to P(T)$ such that the following diagram commutes.
    \index{$\phi$ : Homomorphism}
    \[
       \begin{tikzcd}[
            cells={nodes={text width=4em, align=center}}, row sep=6em, column sep=6em
            ]
            Q \arrow[r, dashed, "\phi"] \arrow[rd, "\gamma" description] & P(T) \arrow[d, "\pi"] \\
            \ProjFree(Y) \arrow[u, "\beta"] \arrow[r, swap, "\alpha"]    & T
        \end{tikzcd}
    \]
    In particular, $\alpha$ factors through $P(T)$.
    Note that $\alpha$ is an epimorphism since $T$ is simple and $\alpha$ is non-zero.
    Since $\pi$ is an essential epimorphism, $\phi \circ \beta \colon \ProjFree(Y) \to P(T)$ is also an epimorphism and hence splits.
    This shows that $P(T)$ is a direct summand of $\ProjFree(Y)$, a contradiction.
    Therefore $\alpha$ is zero, which completes the proof.
\end{proof}

The dual of the previous lemma is the following.
\begin{lem}
    \label{lem:costable}
    If $S$ is a simple right $A$-module and $X$ is a right $A$-module, then
    \[
        \ISHom_A(S, X) \cong \Hom_A(S, \InjFree(X)).
    \]
\end{lem}
\begin{proof}
    From the definition of $\ISHom$ and $\InjFree$,
    \[
        \ISHom_A(S, X) \cong \ISHom_A(S, \InjFree(X)).
    \]
    Hence it suffices to prove that the canonical epimorphism
    \[
        \Hom_A(S, \InjFree(X)) \to \ISHom_A(S, \InjFree(X))
    \]
    is a monomorphism.

    Let $\alpha \colon S \to \InjFree(X)$ factor through an injective right $A$-module $J$ as $\alpha = \gamma \circ \beta$.
    Assume that $\alpha$ is non-zero to obtain a contradiction.
    Let $\iota \colon S \to I(S)$ be the injective hull of $S$.
    \index{$\iota$ : Homomorphism}
    Since $J$ is injective, there exists $\phi \colon I(S) \to J$ such that the following diagram commutes.
    \[
       \begin{tikzcd}[
            cells={nodes={text width=4em, align=center}}, row sep=6em, column sep=6em
            ]
            I(S) \arrow[r, dashed, "\phi"]                                                 & J \arrow[d, "\gamma"] \\
            S \arrow[u, "\iota"] \arrow[r, swap, "\alpha"] \arrow[ru, "\beta" description] & \InjFree(X)
        \end{tikzcd}
    \]
    In particular, $\alpha$ factors through $I(S)$.
    Note that $\alpha$ is a monomorphism since $S$ is simple and $\alpha$ is non-zero.
    Since $\iota$ is an essential monomorphism, $\gamma \circ \phi \colon I(S) \to \InjFree(X)$ is also a monomorphism and hence splits.
    This shows that $I(S)$ is a direct summand of $\InjFree(X)$, a contradiction.
    Therefore $\alpha$ is zero, which completes the proof.
\end{proof}

\section{Proofs}

\subsection*{Tensor-hom adjunction}
Using \cref{lem:hom,lem:tensor}, we prove a tensor-hom adjunction for stable and costable categories.

\begin{proof}[Proof of \cref{main:tensor-hom}]
    Let $\Phi_{XY}$ and $\Psi_{XY}$ be the mutually inverse isomorphisms of vector spaces
    \[
        \begin{tikzcd}[column sep=huge]
            \Hom_B(X \otimes_A M, Y) \arrow[r, shift left, "\Phi_{XY}"] & \Hom_A(X, \Hom_B(M, Y)) \arrow[l, shift left, "\Psi_{XY}"]
        \end{tikzcd}
    \]
    that are natural in $X$ and $Y$ from the ordinary tensor-hom adjunction \cite[Theorem II.4.3]{SY11}.
    \index{$\Phi_{XY}$ : Linear Map}
    \index{$\Psi_{XY}$ : Linear Map}
    We begin by proving that these isomorphisms induce
    \[
        \begin{tikzcd}[column sep=huge]
            \PHom_B(X \otimes_A M, Y) \arrow[r, shift left] & \IHom_A(X, \Hom_B(M, Y)) \arrow[l, shift left].
        \end{tikzcd}
    \]

    \paragraph{$\Phi_{XY}(\PHom_B(X \otimes_A M, Y)) \subseteq \IHom_A(X, \Hom_B(M, Y))$:}
    Assume that $\alpha \colon X \otimes_A M \to Y$ factors through a projective right $B$-module $Q$ as follows.
    \[
        \begin{tikzcd}[
            cells={nodes={text width=4em, align=center}}, row sep=6em
            ]
            X \otimes_A M \arrow[rr, "\alpha"] \arrow[rd, swap, "\beta"] &                              & Y \\
                                                                         & Q \arrow[ru, swap, "\gamma"] &
        \end{tikzcd}
    \]
    Then the following diagram commutes.
    \[
        \begin{tikzcd}[
            cells={nodes={text width=4em, align=center}}, row sep=6em
            ]
            X \arrow[rr, "\Phi_{XY}(\alpha)"] \arrow[rd, swap, "\Phi_{XQ}(\beta)"] &                                                   & \Hom_B(M, Y) \\
                                                                                   & \Hom_B(M, Q) \arrow[ru, swap, "\gamma \circ {-}"] &
        \end{tikzcd}
    \]
    Hence $\Phi_{XY}(\alpha)$ factors through an injective right $A$-module $\Hom_B(M, Q)$ by \cref{lem:hom}.

    \paragraph{$\Psi_{XY}(\IHom_A(X, \Hom_B(M, Y))) \subseteq \PHom_B(X \otimes_A M, Y)$:}
    Assume that $\xi \colon X \to \Hom_B(M, Y)$ factors through an injective right $A$-module $J$ as follows.
    \[
        \begin{tikzcd}[
            cells={nodes={text width=4em, align=center}}, row sep=6em
            ]
            X \arrow[rr, "\xi"] \arrow[rd, swap, "\eta"] &                             & \Hom_B(M, Y) \\
                                                         & J \arrow[ru, swap, "\zeta"] &
        \end{tikzcd}
    \]
    \index{$\xi$ : Homomorphism}
    \index{$\eta$ : Homomorphism}
    \index{$\zeta$ : Homomorphism}
    Then the following diagram commutes.
    \[
        \begin{tikzcd}[
            cells={nodes={text width=4em, align=center}}, row sep=6em
            ]
            X \otimes_A M \arrow[rr, "\Psi_{XY}(\xi)"] \arrow[rd, swap, "\eta\otimes1_M"] &                                                    & Y \\
                                                                                          & J \otimes_A M \arrow[ru, swap, "\Psi_{JY}(\zeta)"] &
        \end{tikzcd}
    \]
    \index{$1_M$ : Homomorphism}
    Hence $\Psi_{XY}(\xi)$ factors through a projective right $B$-module $J \otimes_A M$ by \cref{lem:tensor}.

    The above argument shows an isomorphism of quotient spaces.
    Since homomorphisms $X' \to X$ and $Y \to Y'$ induce $X' \otimes_A M \to X \otimes_A M$ and $\Hom_B(M, Y) \to \Hom_B(M, Y')$, it is straightforward to check the naturality.
    \index{$X'$ : Module}
    \index{$Y'$ : Module}
\end{proof}

\begin{rmk}
    \label{rmk:self-injective}
    A statement similar to \cref{main:tensor-hom} can be found in \cite[Exercise 2.15.6]{Li18} for finite-dimensional self-injective algebras.
\end{rmk}

\subsection*{Reciprocity}
Using \cref{main:tensor-hom,lem:stable,lem:costable}, we prove the reciprocity for projective and injective direct summands.
The proof follows the argument in \cite{Ko26}.
(See also \cite[Lemma 5.2]{Gr02}.)
Let $\Irr(A)$ denote a complete set of representatives of simple right $A$-modules.
\index{$Irr(A)$ @ $\Irr(A)$ : Set}
\begin{proof}[Proof of \cref{main:reciprocity}]
    For a simple right $B$-module $V$ and a simple right $A$-module $U$, set
    \[
        n_{VS} = \mult{P(V)}{S \otimes_A M},
        \qquad
        m_{UT} = \mult{I(U)}{\Hom_B(M, T)}.
    \]
    \index{$V$ : Simple Module}
    \index{$U$ : Simple Module}
    \index{$n_{VS}$ : Natural Number}
    \index{$m_{UT}$ : Natural Number}
    Our theorem reads $n_{TS} = m_{ST}$ with this notation.
    To prove this reciprocity, decompose $S \otimes_A M$ and $\Hom_B(M, T)$ as
    \begin{align*}
        S \otimes_A M &\cong \ProjFree(S \otimes_A M) \oplus \bigoplus_{V \in \Irr(B)} n_{VS} \, P(V), \\
        \Hom_B(M, T)  &\cong \InjFree(\Hom_B(M, T)) \oplus \bigoplus_{U \in \Irr(A)} m_{UT} \, I(U).
    \end{align*}
    Recall that the top of $P(V)$ is $V$ and the socle of $I(U)$ is $U$.
    From the top-socle adjunction \cite[Lemma 2.3]{Sa17}, Schur's lemma and the splitting assumption,
    \begin{align*}
        \dim \Hom_B(P(V), T) &= \dim \Hom_B(V, T) = \delta_{VT}, \\
        \dim \Hom_A(S, I(U)) &= \dim \Hom_A(S, U) = \delta_{SU}
    \end{align*}
    with the Kronecker delta.
    \index{$\delta_{VT}$ : Natural Number}
    Thus
    \begin{align}
        \dim \Hom_B(S \otimes_A M, T) &= \dim \Hom_B(\ProjFree(S \otimes_A M), T) + n_{TS}, \label{eq:n_TS} \\
        \dim \Hom_A(S, \Hom_B(M, T))  &= \dim \Hom_A(S, \InjFree(\Hom_B(M, T))) + m_{ST}.   \label{eq:m_ST}
    \end{align}
    On the one hand, the ordinary tensor-hom adjunction yields
    \begin{equation}
        \label{eq:whole}
        \dim \Hom_B(S \otimes_A M, T) = \dim \Hom_A(S, \Hom_B(M, T)).
    \end{equation}
    On the other hand, since $M$ is skew,
    \begin{align*}
        \Hom_B(\ProjFree(S \otimes_A M), T)
        &\cong \PSHom_B(S \otimes_A M, T)            & & \text{by \cref{lem:stable}} \\
        &\cong \ISHom_A(S, \Hom_B(M, T))             & & \text{by \cref{main:tensor-hom}} \\
        &\cong \Hom_A(S, \InjFree(\Hom_B(M, T)))     & & \text{by \cref{lem:costable}.}
    \end{align*}
    This isomorphism of vector spaces yields
    \begin{equation}
        \label{eq:part}
        \dim \Hom_B(\ProjFree(S \otimes_A M), T)
        = \dim \Hom_A(S, \InjFree(\Hom_B(M, T))).
    \end{equation}
    Hence
    \begin{align*}
        n_{TS}
        &= \dim \Hom_B(S \otimes_A M, T) - \dim \Hom_B(\ProjFree(S \otimes_A M), T) & & \text{by \eqref{eq:n_TS}} \\
        &= \dim \Hom_A(S, \Hom_B(M, T)) - \dim \Hom_A(S, \InjFree(\Hom_B(M, T)))    & & \text{by \eqref{eq:whole}, \eqref{eq:part}} \\
        &= m_{ST}                                                                   & & \text{by \eqref{eq:m_ST}.}
    \end{align*}
    This completes the proof.
\end{proof}

\begin{rmk}
    The splitting assumption in \cref{main:reciprocity} is not essential.
    Indeed, the same argument shows
    \[
        \mult{P(T)}{S \otimes_A M} \times \dim \End_B(T) = \mult{I(S)}{\Hom_B(M, T)} \times \dim \End_A(S)
    \]
    for an arbitrary field $\FF$.
\end{rmk}

\begin{cor}
    \label{cor:Frobenius}
    Let $A$ and $B$ be finite-dimensional Frobenius algebras over a splitting field and let $M$ be a perfect $(A, B)$-bimodule.
    Write $M^\vee = \Hom_B(M, B)$.
    If $S$ is a simple right $A$-module and $T$ is a simple right $B$-module, then
    \[
        \mult{P(T)}{S \otimes_A M} = \mult{I(S)}{T \otimes_B M^\vee}.
    \]
\end{cor}
\index{$({-})^\vee$ : Functor}
\begin{proof}
    Since $A$ and $B$ are Frobenius algebras, perfect bimodules are skew and $\Hom_B(M, T) \cong T \otimes_B M^\vee$ as right $A$-modules.
    Hence the assertion follows from \cref{main:reciprocity}.
\end{proof}

The reciprocity in \cite{Ko26} follows from \cref{cor:Frobenius}.
\begin{cor}
    \label{cor:Koshitani}
    Let $A$ and $B$ be finite-dimensional symmetric algebras over a splitting field and let $M$ be a perfect $(A, B)$-bimodule.
    If $S$ is a simple right $A$-module and $T$ is a simple right $B$-module, then
    \[
        \mult{P(T)}{S \otimes_A M} = \mult{P(S)}{T \otimes_B M^*}.
    \]
\end{cor}
\begin{proof}
    Since $A$ and $B$ are symmetric algebras, $I(S) \cong P(S)$ by \cite[Corollary IV.8.6(i)]{SY11} and $M^\vee \cong M^*$ by \cite[Corollary IV.8.6(ii)]{SY11}.
    Hence the assertion follows from \cref{cor:Frobenius}.
\end{proof}

\section{Examples}
We conclude with examples illustrating \cref{main:reciprocity} for finite-dimensional algebras which are not necessarily symmetric.
\begin{exm}
    Let $A$ and $B$ be finite-dimensional algebras over a splitting field $\FF$.
    Consider a simple right $A$-module $U$ and a simple right $B$-module $V$.
    Write $P(U)^\vee = \Hom_A(P(U), A)$.
    \index{$({-})^\vee$ : Functor}
    Then
    \[
        M = P(U)^\vee \otimes_\FF P(V)
    \] is a skew $(A, B)$-bimodule as
    \begin{align*}
        M \otimes_B B^*
        &\cong P(U)^\vee \otimes_\FF P(V) \otimes_B B^*
        \cong \dim (P(V) \otimes_B B^*) \, P(U)^\vee, \\
        A^* \otimes_A M
        &\cong A^* \otimes_A P(U)^\vee \otimes_\FF P(V)
        \cong \dim (A^* \otimes_A P(U)^\vee) \, P(V).
    \end{align*}
    For a simple right $A$-module $S$ and a simple right $B$-module $T$,
    \begin{align*}
        S \otimes_A M
        &\cong S \otimes_A P(U)^\vee \otimes_\FF P(V)   & & \\
        &\cong \Hom_A(P(U), S) \otimes_\FF P(V)         & & \text{by \cite[Corollary II.4.9]{SY11}} \\
        &\cong \Hom_A(U, S) \otimes_\FF P(V)            & & \\
        &\cong \delta_{US} \, P(V),                     & & \\
        \Hom_B(M, T)
        &\cong \Hom_B(P(U)^\vee \otimes_\FF P(V), T)    & & \\
        &\cong \Hom_\FF(P(U)^\vee, \Hom_B(P(V), T))     & & \\
        &\cong \Hom_\FF(P(U)^\vee, \Hom_B(V, T))        & & \\
        &\cong \delta_{VT} \, I(U)                      & & \text{by $(P(U)^\vee)^* \cong I(U)$.}
    \end{align*}
    Hence we obtain the equality of multiplicities
    \[
        \mult{P(T)}{S \otimes_A M} = \delta_{US}\delta_{VT} = \mult{I(S)}{\Hom_B(M, T)}.
    \]
\end{exm}

The next example shows that skew bimodules are not necessarily perfect.
\begin{exm}
    Let $A$ and $B$ be finite-dimensional algebras over a splitting field $\FF$.
    Consider a semisimple right $A$-module $D$ and a semisimple right $B$-module $E$ defined by
    \[
        D = \bigoplus_{\substack{U \in \Irr(A)\\ \Hom_A(A^*, U) = 0}} U,
        \qquad
        E = \bigoplus_{\substack{V \in \Irr(B)\\ \Hom_B(V, B) = 0}} V.
    \]
    \index{$D$ : Semisimple Module}
    \index{$E$ : Semisimple Module}
    Then
    \[
        N = D^* \otimes_\FF E
    \]
    is a skew $(A, B)$-bimodule as
    \begin{align*}
        N \otimes_B B^*
        &\cong D^* \otimes_\FF E \otimes_B B^*
        \cong \dim (\Hom_B(E, B)) \, D^*
        = 0, \\
        A^* \otimes_A N
        &\cong A^* \otimes_A D^* \otimes_\FF E
        \cong \dim (\Hom_A(A^*, D)) \, E
        = 0.
    \end{align*}
    \index{$N$ : Bimodule}
    For a simple right $A$-module $S$ and a simple right $B$-module $T$,
    \begin{align*}
        S \otimes_A N
        &\cong S \otimes_A D^* \otimes_\FF E \\
        &\cong \Hom_A(S, D)^* \otimes_\FF E \\
        &\cong
        \begin{cases}
            E & \text{if $\Hom_A(A^*, S) = 0$} \\
            0 & \text{otherwise},
        \end{cases} \\
        \Hom_B(N, T)
        &\cong \Hom_B(D^* \otimes_\FF E, T) \\
        &\cong \Hom_\FF(D^*, \Hom_B(E, T)) \\
        &\cong
        \begin{cases}
            D & \text{if $\Hom_B(T, B) = 0$} \\
            0 & \text{otherwise}.
        \end{cases}
    \end{align*}
    Note that if $T$ is projective, then $\Hom_B(T, B) \cong \Hom_B(P(T), B) \neq 0$.
    Similarly, if $S$ is injective, then $\Hom_A(A^*, S) \cong \Hom_A(A^*, I(S)) \neq 0$.
    Hence we obtain the equality of multiplicities
    \[
        \mult{P(T)}{S \otimes_A N} = 0 = \mult{I(S)}{\Hom_B(N, T)}.
    \]
\end{exm}

In particular, the direct sum $M \oplus N$ is a skew bimodule with non-zero multiplicities that need not be perfect.

\section*{Acknowledgments}
The authors are grateful to Markus Linckelmann and Kota Yamaura for their advice.

\printbibliography


\end{document}